\UseRawInputEncoding
\documentclass[12pt]{amsart}

\def\eps{{\varepsilon}}

\def\Cov{{\rm Cov}}

\def\Prob{{\mathbb{P}}}

\def\Var{{\rm Var}}

\def\EXP{{\mathbb{E}}}

\def\bbL{\mathbb{L}}

\def\naturals{\mathbb{N}}

\def\reals{\mathbb{R}}

\def\integers{\mathbb{Z}}

\def\P{{\partial }}

\def\bG{\mathbf{G}}

\def\brC{{\bar C}}

\def\brF{{\bar F}}

\def\brc{{\bar c}}

\def\brk{{\bar k}}

\def\brl{{\bar l}}

\def\brp{{\bar p}}

\def\cA{\mathcal{A}}

\def\cD{\mathcal{D}}

\def\cF{\mathcal{F}}

\def\cS{\mathcal{S}}

\def\cX{\mathcal{X}}

\def\fl{\mathfrak{l}}

\def\fp{\mathfrak{p}}

\def\fq{\mathfrak{q}}

\def\hC{{\hat C}}

\def\tF{{\tilde F}}

\def\tT{{\tilde T}}

\def\ta{{\tilde a}}

\def\tb{{\tilde b}}

\def\tc{{\tilde c}}

\def\tl{{\tilde l}}

\makeatother

\def\beq{\begin{equation}}
\def\eeq{\end{equation}}

\usepackage[usenames,dvipsnames,svgnames,table]{xcolor}
\usepackage{amssymb}

\usepackage[margin=3cm]{geometry}

\newtheorem{theorem}{Theorem}[section]
\newtheorem{proposition}[theorem]{Proposition}
\newtheorem{lemma}[theorem]{Lemma}

\newtheorem{corollary}[theorem]{Corollary}

\theoremstyle{definition}
\newtheorem{remark}[theorem]{Remark}

\numberwithin{equation}{section}

\newcommand{\R} {\mathbb{R}}

\newcommand{\Z} {\mathbb{Z}}
\newcommand{\N} {\mathbb{N}}
\renewcommand{\P} {\mathbb{P}}
\newcommand{\E} {\mathbb{E}}

\newcommand{\lf} {\lfloor}
\newcommand{\rf} {\rfloor}

\def\DS{\displaystyle}
\newcommand{\ignore}[1]{}

\author{Dmitry Dolgopyat}
\address{
Department of Mathematics, University of Maryland, College Park, 
MD 20741, USA}
\email{dmitry@math.umd.edu}

\author{Marco Lenci}
\address{
Dipartimento di Matematica, Universit\`a  di Bologna, 40126 Bologna, Italy; 
Istituto Nazionale di Fisica Nucleare, Sezione di Bologna, 40126 Bologna, 
Italy}
\email{marco.lenci@unibo.it}

\author{P\'eter N\'andori}
\address{
Department of Mathematical Sciences, Yeshiva University, New York, NY 
10016, USA}
\email{peter.nandori@yu.edu}

\title[Global observables for RW: law of large numbers]
{Global observables for random walks: \\ law of large numbers}

\subjclass[2000]{Primary 60F15, 60G50; Secondary 37A40, 60K37}
\keywords{Random walks, law of large numbers, arcsine law, 
random walks in random scenery, global observables.}

\begin{document}

\maketitle

\begin{abstract}
  We consider the sums $\DS T_N=\sum_{n=1}^N F(S_n)$ where $S_n$ is a random walk
  on $\Z^d$ and $F:\Z^d\to \R$ is a global observable, that is, a bounded function
  which admits an average value when averaged over large cubes.
  We show that $T_N$ always satisfies the weak Law of Large Numbers but the strong law
  fails in general except for one dimensional walks with drift. Under additional regularity assumptions on $F$, we obtain the
  Strong Law of Large Numbers and estimate the rate of convergence. The growth exponents which we obtain turn out to be optimal
  in two special cases: for quasiperiodic observables and for random walks in random scenery.
\end{abstract}

\section{Introduction}
\subsection{Motivation.}
Ergodic theory was created in the beginning of the last century motivated by the needs of homogenization
(more specifically the quest to justify the kinetic equations of statistical mechanics). By now ergodic theory is
a flourishing subject. Namely, ergodic theorems are established under very general conditions
and ergodic properties of a large number of smooth systems are known (see e.g. \cite{KH95}).
Moreover, ergodicity turns out to be useful in the questions of averaging and homogenization
(see e.g. \cite{JKO94, LM88, PV81, SVM07}). However, many dynamical systems appearing in applications
preserve infinite invariant measure and ergodic theory of infinite-measure-preserving systems is much less
developed. In fact, most of the work in infinite ergodic theory (see e.g. \cite{A97})
deals with local ($L^1$) observables 
while from physical point of view it is more natural to consider extensive observables (\cite{Kh49, Ru78})
which admit an infinite-volume average. One explanation for this is that while for local observables
ergodic theorems can be obtained with minimal regularity assumptions on the observable,
this is not the case for global observables as the present paper shows.
The study of ergodic properties of infinite measure transformations
with respect to extensive functions started relatively recently \cite{L10}. 
In particular, mixing properties of several systems with respect to global observables 
were obtained in \cite{BGL18, BL19, DN18, L10, L17}. 
A natural question is thus to investigate the law of large numbers
for global observables. A first step in this 
direction was recently taken in \cite{LM18}. In this paper we carry out a 
detailed analysis in the simplest possible setting: random walks on 
$\integers^d$. Our goal is to ascertain the correct spaces for the law of 
large numbers in various cases.

\subsection{Results.}
\label{SSResults}
Let $X_1, X_2,...$ be an iid sequence of $\mathbb Z^d$ valued random variables. 
Let $S_0 = 0$ and $\DS S_N = \sum_{n=1}^{N} X_n$ be the corresponding random walk.
We assume that 
\begin{enumerate}
\item (non-degeneracy) the smallest group supporting the range of $X_1$ is $\mathbb Z^d,$ 
\item (aperiodicity) $g.c.d.\{ n>0: \mathbb P (S_n = 0) > 0 \} = 1.$
\end{enumerate}

We will also assume that $S_n$ is in the normal domain of attraction of 
a stable law
with some index $\alpha$. That is, there is a non-degenerate $d$ dimensional random variable $Y $
such that
$$
\frac{S_n}{n^{1/\alpha }} \Rightarrow Y \text{ if } \alpha \in (0,1) \text{ and }
\frac{S_n - n \E(X_1)}{n^{1/\alpha }} \Rightarrow Y \text{ if } \alpha \in (1,2].
$$
To avoid uninteresting minor technical difficulties, we will mostly assume that $\alpha \neq 1$.

We define several function spaces which proved to be useful in the previous studies of global observables
\cite{DN18}.
Without further notice, we always assume that all functions are bounded.

Given a 
non-empty
subset $V \subset \mathbb Z^d$ and $F \in L^{\infty}(\mathbb Z^d, \mathbb R)$, we write
$$
\bar F_V = \frac{1}{|V|} \sum_{v \in V} F(v).
$$
Given $(a_1, b_1, \dots a_d, b_d)$
with $a_i < b_i$, for $j=1, \dots d$,
let 
\begin{equation}
\label{defV}
V(a_1, b_1, \dots a_d, b_d)=\{x\in {\integers}^d:\; x_j\in [a_j, b_j] \text{ for } j=1, \dots d\}. 
\end{equation}

Let $\bG_+$ be the space of bounded functions on $\integers$ such that the limit
$
\DS \bar F_+ = \lim_{v \to \infty} \bar F_{[0,v]}$
exists
and 
$\bG_-$ be the space of bounded functions on $\integers$ such that the limit
$\DS \bar F_- = \lim_{v \to \infty} \bar F_{[-v,0]}$ exists. Set $\bG_\pm=\bG_+\cap\bG_-.$
Define
$$\bG_0=\{F\in L^\infty(\integers^d, \reals):  \exists \brF \;\; \forall a_1 < b_1, \dots 
a_d < b_d\quad \lim_{L\to \infty} 
\brF_{V(a_1 L, b_1 L, \dots, a_d L, b_d L)}=\brF\}.$$
Note that in dimension 1, $\bG_0=\{F\in \bG_\pm: \brF_+=\brF_-\}.$
Let $\bG_U$ be the space of functions such that for each $\eps$ there is $L$ 
such that for all cubes $V$ with side larger than $L$ we have
\begin{equation}
\label{CubeAve}
|\brF_V-\brF|\leq \eps.
\end{equation}
Let $\bG_\gamma$ be the set of functions where \eqref{CubeAve} only holds if the center of $V$ is within
distance $L^\gamma$ 
of
the origin. Thus $\bG_U \subset \bG_\gamma .$ Also,
$\bG_\gamma\subset \bG_0$ if $\gamma>1.$
Finally, let 
\begin{gather*}
\bG_{\gamma}^{\beta} = \{
F \in L^{\infty}(\integers^d, \reals): \exists \brF\;\;
\forall a_1 < b_1, ..., a_d < b_d\;\;
\exists C: \forall L, 
\forall z \in \integers^d, |z| < L^{\gamma}, \\
|\brF_{z + V(a_1L, b_1L,...,a_dL,b_dL)} - \brF |< C L^{d(\beta -1)}
\}.
\end{gather*}
Clearly, $\bG_{\gamma}^{\beta} \subset \bG_{\gamma} $ for any $\beta <1$.
Also, let $\DS \bG_\infty^\beta=\bigcap_{\gamma>0} \bG_\gamma^\beta.$

Our goal is to study Birkhoff sums
$$ T_N=\sum_{n=1}^N F(S_n) . $$
In particular we would like to know if $\frac{T_N}{N}$ converges to $\brF$ for $F$ in each of the spaces
$\bG_*$ introduced above.

Our results could be summarized as follows.

\begin{theorem}
\label{ThWLLN}
Suppose that $\E(X)=0$ and that
$S_N$ is in the normal domain of attraction of  
a
stable law of some index $\alpha>1.$
Then for all $F\in \bG_0$, $\dfrac{T_N}{N}\Rightarrow \brF$ in law as $N\to \infty.$
\end{theorem}

\begin{theorem}
\label{ThArc}
Suppose that $d=1,$ $\E(X)=0$ and $V(X)<\infty$. Then 
for all $F\in \bG_\pm$, $\dfrac{T_N}{N}$ converges in law as $N\to \infty.$
In particular, if $\brF_-=0$ and $\brF_+=1$ then the limiting law has arcsine 
distribution: for $z\in [0,1]$
 $$ \lim_{N\to \infty} \Prob\left(\frac{T_N}{N}\leq z\right)=\frac{2}{\pi} \arcsin\sqrt{z}. $$
\end{theorem}

Note that Theorem \ref{ThArc} is a simple homogenization result: it says that the limit
distribution of $\frac{T_N}{N}$ remains the same if the oscillatory function $F$ is replaced by
a more regular function
$\brF_- 1_{x<0}+\brF_+ 1_{x\geq 0}$ (see \cite{L39}). 
This confirms the usefulness of global observables in applications.

\begin{theorem}
\label{ThGGamma}
Suppose that $S_N$ is in the normal domain of attraction of the stable law of 
some index $\alpha$. Suppose that either

(i) $1 < \alpha \leq 2,$ $\E(X_1) \ne 0$ and $\gamma>1$ or

(ii) $1 < \alpha \leq 2$, $\E(X_1)=0$ and $\gamma>1/\alpha$ or

(iii) $\alpha\leq 1$ and $\gamma>1/\alpha.$

Then, for all $F\in \bG_\gamma,$ 
$\frac{T_N}{N}\to\brF$ almost surely.
\end{theorem}

\begin{theorem}
\label{ThGGammaBeta}
Suppose $E(X_1) = 0$ and $E(|X_1|^k) < \infty$ for all $k \in \naturals$.
For $d\in \naturals$, let
$$
\rho_d(\beta):=
\begin{cases}
\frac12 & \text{if } \beta \leq \frac{d-1}{d} \\
\frac{d}{2} (\beta -1) + 1 & \text{if } \beta > \frac{d-1}{d},
\end{cases} \quad \quad
\gamma(d, \beta, \eps):=
\begin{cases} \frac{2}{\beta} & \text{if } d=1 \\
\frac{1}{\eps} & \text{if } d\geq 2. \end{cases}
$$
Then for every $d \in \naturals$, for every
$\beta \in [0,1)$ every
$\varepsilon >0$ 
and any $F \in \bG_{\gamma(d, \beta, \eps)}^{\beta}$ with $\brF = 0$, we have
$$\frac{T_N}{N^{\rho_d(\beta)+ \eps} }\to 0 \text{ 
almost surely as }N\to \infty.$$
\end{theorem}

\begin{corollary}
\label{cor:gammainf}
If $d=1$ and $F\in \bG_{2/\beta}^\beta$ or $d\geq 2$ and $F\in \bG_\infty^\beta$ then
with probability 1,
for all $\eps$
$$ \lim_{N\to\infty} \frac{T_N}{N^{\rho_d(\beta)+\eps}}=0. $$
\end{corollary}

\begin{remark}
Let us discuss two special cases.

(A) {\em (Random walk in random scenery)}
If $F(x)$, $x \in \integers^{d}$ are bounded and iid with expectation $0$, 
then by moderate deviation estimates, almost surely for every $\gamma < \infty$ and for
every $\eps >0$, $F \in \bG_{\gamma}^{\frac12 + \eps}$  and $\brF = 0.$ 

Now assuming that the random walker has zero expectation and finite moments of every order, 
Theorem
\ref{ThGGammaBeta} 
implies 
$\frac{T_N}{N^{\frac34 + \eps }}\to 0$ almost surely
in dimension $d=1$. Note that in this case,
$\frac{T_N}{N^{\frac34 }}$ has a non-trivial weak limit by \cite{KS79}.
If $d\geq 2$, Theorem \ref{ThGGammaBeta} gives
 $\frac{T_N}{N^{\frac12 + \eps }}\to 0$ almost surely 
while $\frac{T_N}{N^{\frac12 }}$ 
 ($\frac{T_N}{\sqrt{N\ln N}}$ if $d=2$) has a non-trivial weak limit
(\cite{KS79, B79, B89}). We note that Theorem \ref{ThGGammaBeta} is not new for 
$F$ as above (see \cite{KL98, GP01}) however, we would like to emphasize that
our space $\bG_\gamma^{\frac{1}{2}+\eps}$ includes many more functions than just realizations
of iid process, so both the result and the proof of Theorem \ref{ThGGammaBeta}
are new even for $\bG_\gamma^{\frac{1}{2}+\eps}.$

(B) If $F(x)$ is periodic, $\brF = 0$, then $F \in 
 \bG_\infty^{\frac{d-1}{d}}$. 
Thus, assuming that  the random walker has zero expectation and finite moments of every order,
Corollary
\ref{cor:gammainf} implies
\begin{equation}
\label{AlmDif}
\frac{T_N}{N^{\frac12 + \eps }}\to 0 
\end{equation}
almost surely for all $d$.
Note that by the central limit theorem for finite Markov chains, 
 $\frac{T_N}{\sqrt{N}}$ 
has a Gaussian weak limit.

In fact, our results also give \eqref{AlmDif} for quasi-periodic observables. That is, given 
$d \in \N$ and
a $C^\infty$ function $\mathfrak{F}:\mathbb T^d\to\R,$
let $\hat{\mathfrak F} :\reals^d \to \reals$ be the $[0,1]^d$-periodic extension of $\mathfrak F$. Furthermore,
given
 $d$ vectors $\alpha_{(1)}, \dots \alpha_{(d)}\in \R^d$ 
and an
initial phase $\omega\in [0,1]^d,$ let
$$ 
F(x)= \hat{\mathfrak{F}}\left(\omega+\sum_{j=1}^d x_j \alpha_{(j)} \right)$$
where $(x_1, \dots x_d)$ are coordinates of vector $x\in \Z^d.$
We say that a vector $\alpha \in \integers^d$ is Diophantine, if 
there are constants $K$ and $\sigma$ such that
for each $m\in \Z^d$,
$$ \left|e^{2\pi \langle m, \alpha \rangle }-1\right|\geq \frac{K}{|m|^\sigma}.$$
If $\alpha_{(j)}$ is Diophantine for all $j=1,...,d$,
then $F \in \bG_{\infty}^0$
(see e.g. \cite[\S 2.9]{KH95}) so \eqref{AlmDif} holds.

Thus in both cases (A) and (B) our results give an optimal exponent for the growth rate of $T_N.$
\end{remark}

\begin{remark}
Periodic (and quasi-periodic) observables are special case of stationary ergodic observables.
More precisely, 
let $\mathfrak{T}_1, \dots \mathfrak{T}_d$ be 
commuting measurable maps of a space $\Omega$ preserving a probability measure
$\nu.$ Given a bounded measurable function $\mathfrak{F}$ on $\Omega$ and an initial condition
$\omega\in \Omega$, define
\begin{equation}
\label{StErg}
F_\omega(k)=\mathfrak{F}(\mathfrak{T}^k \omega),
\end{equation}
where for $k=(k_1,\dots k_d)\in \Z^d$ we let 
$\mathfrak{T}^k=\mathfrak{T}_d^{k_d}\dots \mathfrak{T}_1^{k_1}.$ If the family $\mathfrak{T}^k$ is ergodic, 
then the ergodic theorem tells us that for almost all $\omega$, $F_\omega\in \bG_0$ and 
$\bar{F}=\nu(\mathfrak{F}).$ For the observables given by \eqref{StErg} the strong law of large numbers
for almost every $\omega$ follows from ergodicity of the environment viewed by the particle process (\cite{BS}).
Theorem \ref{ThWLLN}
only gives a weak law of large numbers, (except in dimension 1 in the ballistic case,
see Theorem \ref{ThD=1} below). On the other hand our result gives valuable additional information
even for stationary ergodic environments. Namely, the set of full measure where the weak law of large
numbers holds contains all environments where ergodic averages of $\mathfrak{F}$ exist. We also
note that Theorem \ref{ThGGammaBeta}  provides new and non-trivial information even in the stationary ergodic case.
\end{remark}

\begin{theorem}
\label{ThD=1} 
Suppose that $d=1,$ $v=\E(X_1)>0$ and  for all $t\geq 1$, $\Prob(|X_1|>t)\leq C/t^\beta$ for some 
$C>0$ and $\beta>1.$
If $F\in \bG_+$,  then $\frac{T_N}{N}\to\brF_+$ almost  
surely.
\end{theorem}

The next theorem shows that
in general the
strong law of large numbers fails in $\bG_0.$

\begin{theorem}
\label{ThOcean}
Suppose that 
$S_N$ is in the normal domain of attraction of the stable law of some index $\alpha.$
Moreover assume that one of the following assumptions is satisfied

(a) $\alpha>1$ and $\E(X_1)=0;$ or

(b) $\alpha<1.$

Then there exists $F\in \bG_0$ such that, with probability 1, $\frac{T_N}{N}$ does not converge
as $N\to\infty.$
\end{theorem}

\begin{remark}
The same conclusion holds in case (a) even if $\E(X_1)\neq 0.$ However, in this case $\bG_0$ is not an appropriate space to look at since we even do not have a weak law of large numbers in $\bG_0.$
\end{remark}

\section{Weak convergence.}
\label{ScWeak}
Here we prove Theorems \ref{ThWLLN} and \ref{ThArc}.

\subsection{Preliminaries}

First, we recall two useful results.

\begin{theorem}(\cite[Section 50]{GK54})
\label{thm:llt}
Under the assumptions of Theorem \ref{ThWLLN}, 
$S_n$ satisfies the {\it local limit theorem}, i.e. there is a continuous probability density $g$ such that
$$
\lim_{n \to \infty}
\sup_{l \in \mathbb Z^d} | n^{d/\alpha} \mathbb P (S_n = l) - g(l / n^{1/\alpha}) | = 0.
$$
\end{theorem}

\begin{theorem}[Local global mixing, \cite{DN18}]
\label{thm:lgm}
Under the assumptions of Theorem \ref{ThWLLN}, $S_n$ is {\it local global mixing}, i.e.
$$
\lim_{n \to \infty} \EXP(F(S_n)) = \brF.
$$
\end{theorem}

\subsection{Proof of Theorem \ref{ThWLLN}}

Replacing $F$ by $F - \brF$, we can assume that $\brF = 0$.
By Theorem \ref{thm:lgm}, we have $\DS \lim_{N \to \infty} \frac{\E(T_N)}{N} = 0$. 
Thus in order to prove Theorem
\ref{ThWLLN}, it suffices to verify that $\DS \lim_{N \to \infty} \frac{\E(T^2_N)}{N^2} = 0$. 
Let us fix some $\eps >0$ and prove that $\E(T^2_N) < \eps N^2$ for all sufficiently large $N$.
We have
$$
\E(T^2_N) = 2
\sum_{0\leq n_1 < n_2\leq N} \E (F(S_{n_1}) F(S_{n_2}))
+\sum_{n=1}^N \E( F^2(S_n)).
$$
Now writing $\eps_1 = \frac{\eps}{50 \| F\|^2_{\infty}}$, we have
$$
\E(T^2_N) \leq  \frac{\eps}{10} N^2 +
\left|2 \sum_{\eps_1N < n_1 < n_1+\eps_1N < n_2\leq N} \E (F(S_{n_1})F(S_{n_2})) 
\right|.$$
Choose a constant $K$ such that 
$\P (|S_N| > K N^{1/\alpha} /2) < \frac{\eps}{ 20\| F\|^2_{\infty}}$ 
for all sufficiently large $N.$ Thus we have
$$
\E(T^2_N) \leq  \frac{2\eps}{10} N^2 +
\left| 2\sum_{\eps_1N < n_1 < n_1+\eps_1N < n_2\leq N} 
\E (1_{\{|S_{n_1}|, |S_{n_2}| <KN^{1/\alpha}\}}F(S_{n_1})F(S_{n_2}))\right|. $$

By the Markov property, we have
$$
|\E (1_{\{|S_{n_1}|, |S_{n_2}| <KN^{1/\alpha}\}}F(S_{n_1})F(S_{n_2}))| \leq
$$
$$
\sum_{
|x| < KN^{1/\alpha}}
\| F \|_{\infty} \Prob (S_{n_1} = x)
|\E (1_{\{|S_{n_2-n_1}-x| <KN^{1/\alpha}\}}F(S_{n_2-n_1} -x))| .
$$
Thus it is sufficient to prove that for every $x$ with $|x| < KN^{1/\alpha}$,
\begin{equation}
\label{eq:towerrule}
|\E (1_{\{|S_{n_2-n_1}-x| <KN^{1/\alpha}\}}F(S_{n_2-n_1} -x))|  
\leq \frac{\eps}{5 \| F \|_{\infty}}.
\end{equation}

Recall that $g$ is the density function of the limiting distribution of $S_n / n^{1/\alpha}$ as in Theorem 
\ref{thm:llt}.
Now we choose $\delta$ so that the oscillation of $g$ on any cube
of side length $ \delta/\eps_1^{1/\alpha}$ 
within distance 
$2K / \eps_1^{1/\alpha}$ from the origin
 is less than 
\begin{equation}
\label{DefEta}
 \eta := \frac{\eps \eps_1^{d/\alpha}}{20 K^d \| F\|_{\infty}^2}.
\end{equation}
 
Let us partition the box $[-K,K]^d \subset \reals^d$ into boxes 
of side length $\delta$. Denote these boxes by $B_k, k=1,...,\brk$ and let 
$B_{k,M} = M B_k \cap \integers^d$ for positive integers $M$.
Since, $F\in \bG_0$ we obtain from the definition of $\bG_0$ that
\begin{equation}
\label{eq:mesobox}
\lim_{M \to \infty} \brF_{B_{k,M}} = 0.
\end{equation}
Note that the convergence in \eqref{eq:mesobox} is uniform in $k$ as there
is a finite number of $k$'s.

Now let $m= n_2-n_1$ and $M= N^{1/\alpha}$. Then 
$m^{1/\alpha} > \eps_1^{1/\alpha} M$.
Recalling \eqref{eq:towerrule}, for every $x\in \integers^d$ with $|x|< KM$,
we have
\begin{equation}
\label{eq:weak2}
| \E(1_{\{|S_{m} + x| <KM\}}F(S_{m} + x)) | \leq
\sum_{k=1}^{\brk} |
\E(1_{\{S_{m} + x \in B_{k,M}\}}F(S_{m} + x))|
\end{equation}

Now by Theorem \ref{thm:llt} and the choice of $\delta$, we have
for all $x\in \integers^d$ with $|x|< K M$
\begin{equation}
\label{eq:weak3}
| \E( 1_{\{ S_{m} +x \in B_{k,M} \}}F(S_{m} + x)) |= 
| \sum_{y \in B_{k,M}} (p_{k,x,m} + e_{x,k,m,y})m^{-d/\alpha} F(y) |,
\end{equation}
where $p_{k,x,m}=g\left(\frac{z_k M -x}{m^{1/\alpha}}\right)$,
$z_k$ is the center of $B_k$ and  
for $m$ sufficiently large $e_{x,k,m,y}< 2 \eta$ 
(where 
$\eta$ is given by \eqref{DefEta}) 
uniformly in $x,k,y$ as above.
Consequently, 
$$\sum_k \sum_y e_{x,k,m,y} m^{-d/\alpha}
F(y)
\leq 
m^{-d/\alpha} (KN^{1/\alpha} )^d 2 \eta 
\|F\|_{\infty}
\leq  \frac{\eps}{10\|F\|_{\infty}}$$ 
for sufficiently large $m$ (that is for sufficiently large $N$). Thus dropping $e_{x,k,m,y}$ from the right hand side of
\eqref{eq:weak3} gives a negligible error.
The remaining term is $ \DS p_{k,x,m}m^{-d/\alpha} \sum_{y \in B_k} F(y)$, which 
when summed over $k$, is small by \eqref{eq:mesobox}. Thus the absolute value of
\eqref{eq:weak2} is smaller than $\eps /(5 \| F\|_{\infty})$ for $N$ sufficiently large, 
which completes the proof of \eqref{eq:towerrule}. Theorem \ref{ThWLLN} follows.

\subsection{Proof of Theorem \ref{ThArc}}

We prove the second statement. The first one is a trivial corollary. Indeed, 
given $F \in \bG_{\pm}$ with $\brF_- = \brF_+$,
the convergence follows from Theorem \ref{ThWLLN}. On the other hand if $\brF_- \neq \brF_+$, then we can
consider
$\DS \tilde F(x) = \frac{F(x) - \brF_-}{\brF_+ - \brF_-}$ and note that 
$\bar{\tilde F}_-=0$, $\bar{\tilde F}_+=1$ and
$\DS \tilde{T}_N=
\sum_{n=1}^N \tilde F(S_n)= \frac{T_N-N \brF_-}{\brF_+-\brF_-},$
whereby one derives the limit distribution of $T_N$.

Thus we assume $\brF_- =0 $, $\brF_+ =1$.
Denote $H(x) = 1$ if $x>0$, otherwise $H(x) = 0$. Let us also write 
$G(x) = F(x) - H(x)$.

Then 
$\bar G_- =\bar G_+= 0$  and so
$G \in \bG_0$. 
Decompose
$$
T_N = \frac{1}{N} \sum_{n=1}^N G(S_n) +
\frac{1}{N} \sum_{n=1}^N H(S_n).
$$
The first sum on the RHS converges to zero by Theorem \ref{ThWLLN} and
the second sum on the RHS converges weakly to the arcsine law by classical theory
\cite[\S XII.8]{Fel}. Then the LHS also converges weakly to the arcsine law
by Slutsky's theorem.

\section{SLLN in $\bG_\gamma.$}
Here we prove Theorem \ref{ThGGamma}.

Fix $F \in \bG_\gamma$. As before, we can assume w.l.o.g.\ that 
$\bar F = 0$. 
Since $F \in \bG_\gamma$, the proof of Theorem\ 2.3 of \cite{DN18} shows that,
for any $\eta \in (0,1)$,

\begin{equation} \label{GGamma10}
  \lim_{k \to \infty} \, \sup_{|x| \le k^{\eta\gamma}} \left| 
  \frac{\E_x(T_k)} {k} \right| = 0,
\end{equation}
where $\E_x$ denotes the expectation in the case where $S_0=x$.

\ignore{We start by proving the theorem for case $0 < \alpha < 1$, or 
$1 < \alpha < 2$ and $E(X_1)=0$. Fix $\eta$ as in (\ref{GGamma10})
such that $\eta \gamma > 1/\alpha$ and define
\begin{equation} \label{GGamma15}
  \cA_N := \left\{ \max_{0 \le k \le N} |S_k| \le 
  N^{\eta \gamma} \right\}.
\end{equation}}

\begin{lemma}
\label{LmMax}
Suppose that $1< \gamma_1 $ in case (i) or
$1/\alpha<\gamma_1$ in cases (ii) and (iii). Then with probability one 
 we have that, for large $N$, 
$$ \max_{0\leq k \leq N} |S_k|\leq N^{\gamma_1}. $$
\end{lemma}

\begin{proof}
In case (i) the statement follows from the Law of Large Numbers, so we only need to consider
cases (ii) and (iii).
We have for any $\eps >0$ that $|S_N| >N^{1/\alpha + \eps}$ holds only finitely many times almost surely
by \cite{M39} in case (ii) and by \cite{L31} in case (iii). 
\end{proof}

Choose $\gamma_1<\gamma$ as in Lemma \ref{LmMax}
and
$\eta<1$ such that 
$\gamma_1<\gamma\eta.$ 
 For
$j = 0, 1, \ldots, \lf N^{1-\eta} \rf$, set $\tilde T_j := \tilde T_{N,j} :=
N^{-\eta} \, T_{\lf j N^\eta \rf}$ (with the convention $T_0 \equiv 0$) 
and denote by $\tilde \cF_j := \tilde \cF_{N,j}$  the $\sigma$-algebra 
generated by $\{ S_k \}_{k=0}^{\lf j N^\eta \rf}$. 
Denote $\cA_{j,N}=\{|S_{\lf j N^\eta \rf}|\leq N^{\gamma_1}\}.$
Fix $\eps>0$.
We claim that there exists $N_0 = N_0(\eps)$ such that, for all $
N \ge N_0$ and $j < \lf N^{1-\eta} \rf$,
\begin{equation} \label{GGamma20}
 \left| \E \! \left( \left. 1_{\cA_{j,N}} \! \left( \tilde T_{j+1} - \tilde T_j \right) \right| 
  \tilde \cF_j \right) \right| \le \eps.
\end{equation}
Indeed if $\cA_{j,N}$ occurs then \eqref{GGamma20} holds due to \eqref{GGamma10},
otherwise it holds since the LHS is zero.

\ignore{
To see this we use (\ref{GGamma10}) with $k := N^\eta$. For all $N$ larger
than or equal to some $N_0(\eps)$ and $|x| \le N^{\eta \gamma}$, we have 
that $| N^{-\eta} E_x(T_{N^\eta}) | \le \eps$. All realizations of the RW 
within $\cA_N$ are such that $| S_{\lf j N^\eta \rf} | < N^{\eta \gamma}$.
So the Markov property of the RW proves (\ref{GGamma20}). Moreover 
$| \tilde T_{j+1} - \tilde T_j | \le \| F \|_\infty$.}

Setting
\begin{equation} \label{GGamma30}
  Y_j := 1_{\cA_{j,N}} \! \left( \tilde T_{j+1} - \tilde T_j \right) - D_j
\end{equation}
where
\begin{equation} \label{GGamma40}
  D_j:=\E \! \left( 1_{\cA_{j,N}} \! \left( \tilde T_{j+1} - \tilde T_j \right) \Big| 
  \tilde \cF_j \right)
\end{equation}
defines a martingale difference, w.r.t.\ $\{ \tilde \cF_j \}$, with $|Y_j| \le
\| F \|_\infty + \eps$. Applying Azuma's inequality we get that, for all
 $\delta>0$,
$$
  \Prob \!\left( \left| \sum_{j=0}^{N^{1 - \eta} -1} Y_j \right| \ge \delta N^{1 - \eta}
  \right) \le 2 \exp \!\left( -\frac{\delta^2 N^{1 - \eta}} {2 (\| F \|_\infty + 
  \eps)^2} \right).
$$

Therefore, by Borel-Cantelli,
$$
  \limsup_{N \to \infty} \, \frac1 {N^{1 - \eta}} \! 
  \left| \sum_{j=0}^{N^{1 - \eta} -1} Y_j \right|
  \le \delta \quad \mbox{a.s.}
$$
Since $\delta$ is arbitrary, 
with probability one we have
\begin{equation}
\label{LLNYJ}
\lim_{N\to\infty} \frac1 {N^{1-\eta}} \!\! \sum_{j=0}^{N^{1 - \eta} -1} Y_j =0. 
\end{equation}
On the other hand, 
definitions (\ref{GGamma30})--(\ref{GGamma40}) and Lemma \ref{LmMax} show 
that, with probability one, for all large $N$ depending on the realization of the walk,

\begin{equation}
\label{NoExcNoGaps}
T_N=T_{\lfloor N^{1-\eta} \rfloor N^{\eta}} +O(N^{\eta}) = 
N^\eta \left(\sum_{j=0}^{N^{1 - \eta} -1} Y_j +\sum_{j=0}^{N^{1-\eta} -1} D_j \right)
+ O(N^{\eta}).
\end{equation}
In view of \eqref{GGamma20}, \eqref{LLNYJ}, and \eqref{NoExcNoGaps}
 we have:
$$  \limsup_{N \to \infty}\,  \left|\frac{T_N}{N} \right|
 = \limsup_{N \to \infty} \, \frac{1}{N^{1 - \eta}} \! 
 \left| \sum_{j=0}^{N^{1 - \eta} -1}  \left(Y_j+D_j\right)  \right|$$
 $$\hskip30mm=\limsup_{N \to \infty} \, \frac{1}{N^{1 - \eta}} \! 
 \left| \sum_{j=0}^{N^{1 - \eta} -1} D_j  \right|
   \le \eps \quad \mbox{a.s.}$$
Since $\eps$ is arbitrary, $\DS \lim_{N\to\infty} \frac{T_N}{N}=0$
almost surely. 


\section{Speed of convergence in $\bG^{\beta}_\gamma.$}
\label{ScSpeed}
Here we prove Theorem \ref{ThGGammaBeta}.

Note that $\bG^{\beta_1}_{\gamma_1} 
\subset \bG^{\beta_2}_{\gamma_2}$ whenever $\beta_1 \leq \beta_2$ and $\gamma_2 \leq \gamma_1$.
Since $\rho_d(\beta)$ is constant for $\beta \in [0,(d-1)/d]$ and is continuous at $(d-1)/d$, it is sufficient to prove the theorem
for 
\begin{equation}
\label{betabound}
\beta > \frac{d-1}{d} .
\end{equation}

Let $\mathbb P_x(.) = \mathbb P(.|S_0 = x)$, $\mathbb E_x(.) = \mathbb E(.|S_0 = x)$.

We start with the following
\begin{proposition}
\label{lemma:weak}
Under the conditions of Theorem \ref{ThGGammaBeta},
for every $d \in \naturals$, every $\beta \in ((d-1)/d,1)$ and every $\eps >0$ there exists some $\delta >0$ so that 
\begin{equation}
\label{eq:lemmaweak}
\sup_{x_0: |x_0| \leq N^{1/2 + \delta}}\mathbb  E_{x_0} (T_{N}^2) < C N^{2\rho_d(\beta) + 2\eps}.
\end{equation}
\end{proposition}

Note that Proposition \ref{lemma:weak} combined with Chebyshev's inequality implies that
$$
 \frac{T_N}{N^{\rho_d(\beta)+ \eps} }\Rightarrow 0
\text{ in law as } N\to \infty .
$$

Section \ref{ScSpeed} is divided into three parts. 
In \S \ref{SSRefl} we derive Theorem \ref{ThGGammaBeta} from Proposition \ref{lemma:weak}.
In \S \ref{SSVar-1} we prove Proposition \ref{lemma:weak}
for $d=1$. In \S \ref{SSVar-2},  we extend the proof of Proposition \ref{lemma:weak} to arbitrary dimension $d$.

\subsection{Proof of Theorem \ref{ThGGammaBeta}}
\label{SSRefl}

Here, we derive the theorem from Proposition \ref{lemma:weak}. For simplicity we write $\rho = \rho_d(\beta)$.

We will show that 
\begin{equation}
\label{eq:bc}
\P (\exists n \leq N: |T_n| > 2 N^{\rho + \eps /2}) \leq  C N^{- \eps /2}
\end{equation}

If \eqref{eq:bc} holds, then writing $N_k = 2^k$, we find 
$$
\P (\exists n = N_{k-1},... N_k : |T_n| > 2 N_k^{\rho + \eps /2}) \leq  C N_k^{- \eps /2}.
$$
and the theorem follows from Borel Cantelli lemma.
To prove \eqref{eq:bc}, let us write 
$$\tau_N = \min \{ \min \{ n: |T_n| > 2 N^{\rho + \eps /2} \}, N \}.$$
Then
\begin{gather*}
\P (\exists n \leq N: |T_n| > 2 N^{\rho + \eps /2}) \\
\leq
\P (|T_N| > N^{\rho + \eps /2}) + \P (|T_{\tau_N}| > 2 N^{\rho + \eps /2}, |T_N| \leq N^{\rho + \eps /2}) = : p_1 + p_2.
\end{gather*}
By Proposition \ref{lemma:weak} for $x_0 = 0$ and by Chebyshev's inequality, 
we have $p_1 \leq C N^{- \eps /2}$. To bound $p_2$, we distinguish two cases: $S_{\tau_N} > N^{1/2 + \delta}$ and 
 $S_{\tau_N} \leq N^{1/2 + \delta}$. 
 The first case has negligible probablity by moderate deviation bound for random walks (see
 formula \eqref{MDE} below). In the second case we compute
\begin{gather*}
 \P (|T_{\tau_N}| > 2 N^{\rho + \eps /2}, |T_N| \leq N^{\rho + \eps /2}, |S_{\tau_N}| \leq N^{1/2 + \delta})\\
\leq \sup_{x_0 : |x_0| \leq  N^{1/2 + \delta} } \max_{n = 1,..., N} \P_{x_0}( |T_n| \geq N^{\rho + \eps /2})
\end{gather*}
which is again bounded by $ C N^{- \eps /2}$ by Proposition \ref{lemma:weak} and Chebyshev's inequality.
We have verified \eqref{eq:bc} and finished the proof of the theorem.

\subsection{Proof of Proposition \ref{lemma:weak} for $d=1$}
\label{SSVar-1}

We have $\beta \in (0,1)$ and $2 \rho_d(\beta) = \beta +1$.
We start by recalling 
some results on expansions in the LLT in case all moments are finite (a.k.a. Edgeworth expansion).

\begin{theorem}(\cite[Section 51]{GK54})
\label{thm:llt2}
Under the assumptions of Theorem \ref{ThGGammaBeta}, there are polynomials $Q_1, Q_2, ...$
so that for any $M \in \N$ 
\begin{equation}
\label{LLTerror}
\P (S_n = l) = \frac{1}{\sqrt n} \mathfrak g (l /\sqrt n) 
(1 + \sum_{m=1}^M Q_m(l /\sqrt n) n^{-m/2}) + e_{n,l,M}
\end{equation}
where $\mathfrak g$ is a Gaussian density and
$$
\limsup_{n \to \infty} \sup_{l \in \integers} e_{n,l,M} n^{M/2 +1} < \infty.
$$
\end{theorem}

Next, we claim that for any $\eta>0$ there exists $C<\infty$ such that 
\begin{equation}
\label{MDE}
\P (|S_n| > n^{1/2 + \eta}) < C n^{- 1/\eta}.
\end{equation}
To prove this claim, first observe that 
$$
\mathbb E (S_n^{2k}) = \sum_{n_1 = 1}^{n} ... \sum_{n_{2k} = 1}^{n} \mathbb E (X_{n_1} ... X_{n_{2k}}) \leq C_{2k} n^{k}. 
$$
Indeed, since $\mathbb E(X_1) = 0$, for every $i=1,...,2k$ we need to have $j \neq i$ so that $n_i = n_j$ to make the expectation non zero.
This estimate for $k > 1/(2 \eta ^2)$ combined with the Markov inequality gives \eqref{MDE}.

Given a function $h = h(n,x) : \integers_+ \times \integers \to \reals$, we write 
${ \nabla}h$ for the discrete
derivative in the second coordinate, i.e. 
$$\nabla h(n,x) = h(n,x) - h(n, x-1). $$
Note that 
\begin{equation}
\label{productrule}
\nabla (gh)(n,x) = (\nabla g)(n,x)h(n,x) + g(n, x-1)(\nabla h)(n,x).
\end{equation}
We will also write $\nabla^k h$ for the $k$th discrete derivative.

Denote
\begin{equation}
\label{defH}
H(n,x) = \P(S_n = x).
\end{equation}
With this notation, \eqref{MDE} can be rewritten as
\begin{equation}
\label{MDE2}
\sum_{x: |x| \geq n^{1/2+ \eta}} H(n,x) < C_{\eta} n^{-1/\eta} \text{ for any } \eta >0.
\end{equation}

Also \eqref{LLTerror} implies that there is a constant $c$ so that,
for every $k=0,1,2$,
\begin{equation}
\label{disder}
\sup_{x \in \integers}|\nabla^k H (n,x)| \leq c n^{-\frac{k+1}{2}}.
\end{equation}



Observe that
\begin{equation}
\label{eq:moment}
\mathbb E_{x_0}({T_{N}^2}) = \sum_{0 \leq n_1\leq n_2 \leq N} c_{n_1,n_2} E_{n_1,n_2}(x_0)\
\end{equation}

where $c_{n_1,n_2} = 1$ if $n_1 = n_2$
and $c_{n_1,n_2} = 2$ otherwise and 
\begin{gather*}
E_{n_1,n_2}(x_0) = \mathbb E_{x_0}(F(S_{n_1}) F(S_{n_2})) \\
= \sum_{x_1, x_2 \in \integers}  \P(S_{n_2 - n_1} = x_2 - x_1) 
\P_{x_0}(S_{n_1} = x_1) 
F(x_1) F(x_2)
\end{gather*}

We will show the following: for any $0 \leq n_1 \leq n_2 \leq N$ such that
\begin{equation}
\label{gapsqrt} n_1>N^\alpha, \quad
n_2 - n_{1} > { N^{\alpha} \text{ where } \alpha = 1/ \gamma = \beta / 2}
\end{equation} 
and for any $x_0$ with $|x_0| < N^{1/2 + \delta}$, we have
\begin{equation}
\label{KSK3}
|E_{n_1,n_2}(x_0)| \leq 
\end{equation}
$$
C n_{1}^{\frac{\beta - 1}{2} + \eps}
(n_{2}-n_{1})^{\frac{\beta - 1}{2} + \eps} 
+
C n_{1}^{\frac{\beta }{2} + \eps}
(n_{2}-n_{1})^{\frac{\beta - 2}{2} + \eps}.
$$

Summing the estimate \eqref{KSK3} for $n_1,n_2$ satisfying \eqref{gapsqrt} we obtain $N^{\beta +1 + 2 \eps}$ as needed.
To complete the proof of the proposition,
it remains to 

(I) prove \eqref{KSK3};

(II) verify that the contribution of $(n_1,n_2)$'s that do not satisfy \eqref{gapsqrt} is also negligible.

We start with (I).


We will use the following lemma:
\begin{lemma}
\label{lem:KS}
There is a constant $\hC$ such that for any positive integer $n$ and any constants
$A,B$, the following holds. If a bounded function $g(x): \integers \to \reals$ satisfies 
\begin{itemize}
\item[(H1)] $\DS \sup_{x: |x| \leq n^{1/2+ \delta'}} |g(x)| \leq A$
\item[(H2)] $\DS \sup_{x: |x| \leq n^{1/2+ \delta'}}|\nabla g(x)| \leq B$
\end{itemize}
for some sufficiently small $\delta'$, then, for $i=0,1$,
\begin{equation}
\label{KSlemma}
\sup_{y: |y| \leq \frac12 n^{\frac{1 /2 + \delta'}{\alpha}}}
\left| \sum_{z \in \integers} \nabla^i H(n,z-y) g(z) F(z) \right| \leq \hC ( \|g\|_{\infty} n^{-10}+ 
A n^{\frac{\beta - i- 1}{2} + \eps} + B n^{\frac{\beta -i}{2} + \eps} ).
\end{equation}
\end{lemma}

Note that $y$ is allowed to be of order $n^{\frac{1 /2 + \delta'}{\alpha}} \gg n^{1 /2 + \delta'}$ which is the range for $x$
in the hypothesis.
This is important in the application of the lemma later, especially when $n_1$ is big and $n_2 - n_1$ is small.

\begin{proof}
For the rest of the section $C$ will denote a constant
(independent of $A$ and $B$) whose value may change from line to line.
By \eqref{MDE2} and since $F$ is bounded, the sum for $z$'s with $|z-y| > n^{1/2 + \delta'}$
is bounded by $C \|g\|_{\infty} n^{-11}$.
Denote $\DS I(x) = \sum_{w=y- 2 n^{1/2 + \delta'}}^x F(w)$. 
Using summation by parts and (H1),
we find
\begin{gather}
\sum_{z:\; |z-y| \leq n^{1/2 + \delta'}} \nabla^i H(n,z-y)g(z) F(z)
= 
O\left(An^{- 10 }\right) + \nonumber \\
- 
\sum_{z :\; |z -y| \leq n^{1/2 + \delta'}} I(z -1)  \nabla_z \left(\nabla^i H(n,z-y)g(z)\right).
\label{KSlemma1.2}
\end{gather}

Using \eqref{productrule}, \eqref{disder}, (H1) and (H2) we find that
\begin{equation}
\label{eq:prod}
\nabla_z \left(\nabla^i H(n,z-y)g(z)\right) 
\leq C( A n^{-\frac{i+2}{2}} + B n^{-\frac{i+1}{2}}).
\end{equation}
Next, we estimate $|I(z)|$ for $z$ satisfying $|z-y| \leq n^{1/2 + \delta'}$. For any such $z$, $I(z)$ is defined as a sum over the interval
$[a,z] \cap \integers$, where $a = y - 2 n^{1/2 + \delta'}$ and consequently 
the length of this interval satisfies $L = z-a  \in [ n^{1/2 + \delta'}, 3 n^{1/2 + \delta'}]$. 
Using $\alpha = 1/\gamma < 1$ and our assumption $|y| \leq \frac12 n^{\frac{1/2 + \delta'}{\alpha}}$, we find that 
$|a| < L^{\gamma}$ for $n$ sufficiently large.
Since $F \in \bG_{\gamma}^{\beta}$,
with $\brF = 0$, we use the definition of $\bG_{\gamma}^{\beta}$ to conclude
\begin{equation}
\label{gammaapp}
|I(z)| \leq CL^{\beta} \leq C n^{\frac{\beta}{2} + \eps}
\end{equation}
(assuming that $\delta' = \delta'(\eps)$ is small enough).
The last two estimates imply that the sum in \eqref{KSlemma1.2}
is bounded by $C (A n^{\frac{\beta - 1 -i }{2} + \eps} + B n^{\frac{\beta -i }{2} + \eps} )$.
\end{proof}

Now we are ready to estimate $E_{n_1,n_2}(x_0)$. First, let
\begin{equation}
\label{DefG1}
g_{1}(x_{1}) := g_{1,n_2 - n_{1}}( x_{1}) := \sum_{x_2 \in \integers} 
H({n_2 - n_{1}}, x_2 - x_{1}) F(x_2).
\end{equation}

By definition, $\| g_1 \|_{\infty}\leq \|F\|_{\infty}$.
Applying Lemma \ref{lem:KS} with $i=0$, $g = 1$, $n ={n_{2}-n_{1}}$, $A=1$, $B=0$
and using $n_2 - n_{1} > N^{\alpha}$, we find
\begin{equation}
\label{eqa1}
\sup_{x_1: |x_1|\leq \frac12 N^{1/2 + \delta'}} |g_{1}(x_1)| \leq C(n_{2}-n_{1})^{\frac{\beta - 1}{2} + \eps}.
\end{equation}
Using Lemma \ref{lem:KS} the same way but now with $i=1$,  we find
\begin{equation}
\label{eqb1}
\sup_{x_1: |x_1|\leq \frac12 N^{1/2 + \delta'}} |\nabla g_{1}(x_1)| \leq C(n_{2}-n_{1})^{\frac{\beta - 2}{2} + \eps}
\end{equation}

Next, set
\begin{gather}
\label{DefG2}
g_{2}(x_{0}) := g_{2,  n_1 }( x_{0}) 
:=
\sum_{x_{1} \in \integers} 
H({n_{1} - n_{0}}, x_{1} - x_{0}) 
g_{1}( x_{1}) F(x_{1})
\end{gather}

Now we use Lemma \ref{lem:KS} with $i=0$,
$n= n_{1} $, $g=g_1$, 
$A = (n_2 - n_{1})^{\frac{\beta - 1}{2} + \eps}$, $B=(n_2 - n_{1})^{\frac{\beta - 2}{2} + \eps}$.
Since $n \leq N$, 
\eqref{eqa1} and \eqref{eqb1} give
(H1) and (H2) (with $\delta'/2$ instead of $\delta'$).
Also using that $\|g_1\|_{\infty} \leq \|F\|_{\infty}^2$ and $n_{1} - n_{2} > N^{\alpha}$, we get
\begin{equation}
\label{KSK4}
\sup_{x_0: |x_0|\leq \frac12  N^{1/2 + \delta'/2 }} |g_{2}(x_0)|
 \leq 
\end{equation}
$$
C n_{1}^{\frac{\beta - 1}{2} + \eps}
(n_{2}-n_{1})^{\frac{\beta - 1}{2} + \eps} 
+
C n_{1}^{\frac{\beta }{2} + \eps}
(n_{2}-n_{1})^{\frac{\beta - 2}{2} + \eps}
$$
which gives \eqref{KSK3} (with $\delta = \delta'/4$).

It remains to verify (II), that is that the contribution of pairs $(n_1,n_2)$'s that do not satisfy \eqref{gapsqrt} is negligible. 

First, assume that $n_1 > N^{\alpha}$ and $n_2 - n_1 { \leq} N^\alpha$. Then we derive as in \eqref{KSK4} but using the 
trivial bounds $A= 1 + \| F\|_{\infty}$, $B = 2(1 + \| F\|_{\infty})$ 
that 
\begin{equation*}
\sup_{x_0: |x_0|\leq N^{1/2 + \delta}} |g_{2}(x_0)|
 \leq 
C n_{1}^{\frac{\beta - 1}{2} + \eps}
+
C n_{1}^{\frac{\beta }{2} + \eps}.
\end{equation*}
Summing this estimate for $n_1 = N^{\alpha}, ..., N$ and multiplying by 
$N^{\alpha}$ for the number of choices of $n_2$, we obtain
\begin{equation}
\label{eq:smalln}
O(N^{\alpha} N^{\frac{\beta + 1}{2}})  = O( N^{\beta + \frac12}) 
=o( N^{\beta + 1}). 
\end{equation}

Next, assume that $n_1 <N^\alpha,$ $n_2 - n_1 < N^\alpha$. Using the bound 
$|E_{n_1,n_2}(x_0)| \leq \| F \|_{\infty}^2$ we obtain
$$
\sum_{n_1=0}^{N^\alpha}\sum_{n_2=n_1}^{n_1 + N^\alpha}|E_{n_1,n_2}(x_0)| 
\leq C N^{2 \alpha} = CN^{\beta} 
=o( N^{\beta +1} ). 
$$

Finally, assume that  $n_1 \leq N^{\alpha}$ and $n_2 - n_1 > N^\alpha.$
By \eqref{MDE2}, we can assume that $S_{n_1} - S_{n_0} \leq N^{1/2 + \delta}$. 
Then \eqref{eqa1} still holds and we
conclude that
\begin{equation}
\label{eq:smalln1}
|E_{n_1, n_2}(x_0)| \leq C (n_2-n_1)^{\frac{\beta -1}{2} + \eps}.
\end{equation}
 Summing for $n_2  = n_1+N^{\alpha}, ..., N$ and multiplying by $N^\alpha$, we obtain
the same error term as in \eqref{eq:smalln}.

This completes the proof.

\subsection{Proof of Proposition \ref{lemma:weak} for $d\geq 2$}
\label{SSVar-2}

In dimension $d$, we have
\begin{equation}
\label{disder2}
\sup_{x \in \integers^d}|\nabla_{i_1} ... \nabla_{i_k} H (n,x)| \leq c n^{-\frac{k+d}{2}}.
\end{equation}
for any $i_1,...i_k =1,...,d$, where $\nabla_i$ denotes the discrete derivative 
with respect to $x_i$, the $i$-th component of $x$.
We apply a similar approach to the case
$d=1$. That is, we perform summations by parts to estimate 
$g_1$ and $g_2$ defined by \eqref{DefG1} and \eqref{DefG2}.
Each time we need $d$ summations by parts. For example, if $d=2$, then
with $m = n_2 - n_{1}$,
\begin{gather}
g_{1,m}(0) = \sum_{|x|< m^{1/2+\delta}} \sum_{|y|<m^{1/2+\delta}} H(m, (x,y)) F(x,y)\nonumber \\
\approx \sum_{|x|< m^{1/2+\delta}} \sum_{|y|<m^{1/2+\delta}}
\nabla_{2} H(m,(x,y)) I_1(x,y) \nonumber \\
\approx  \sum_{|x|< m^{1/2+\delta}} \sum_{|y|<m^{1/2+\delta}}
\nabla_{1} \nabla_2 H(m,(x,y)) I(x,y) \label{eq:2dimcomp}
\end{gather}
where 
$$I_1(x,y) = \sum^y_{z=-m^{1/2 + \delta}} F(x,z) \text{ and }
I(x,y) = \sum^x_{w=-m^{1/2 + \delta}} 
\sum^y_{z=-m^{1/2 + \delta}} F(w,z)$$
and $a_m \approx b_m$ means that the $a_m - b_m$ is superpolynomially small in $m$. 
Recalling that $F \in \bG_{1/\eps}^{\beta}$ and $\brF=0$, we have
$|I(x,y)| \leq Cm^{(1/2 + \delta)2\beta}$. Substituting this estimate and \eqref{disder2}
into \eqref{eq:2dimcomp}, we find
$$
|g_{1,m}(0)| \leq \sum_{|x|< m^{1/2+\delta}} \sum_{|y|<m^{1/2+\delta}} C m^{-2}m^{\beta + 2\delta \beta}
\leq Cm^{\beta - 1 + \eps}.
$$
assuming $2 \delta (1 + \beta) < \eps$.
Using
that $|I(x,y)| \leq C m^{\beta(\frac12 + \delta)}$, we find
$$
|g_{1,m}(0)| \leq Cm^{\beta - 1 + \eps}.
$$

To simplify formulas, we will use the notation
$$
a_N \lesssim b_N \text{ if } a_N \leq C b_N N^{\eps}.
$$

\ignore{
Set 
\begin{equation}
\label{defalpha1}
\alpha_1 = \frac{\beta}{2-\beta}
\end{equation}
Proceeding as in case $d=1$, we find that  if $\gamma > (2 - \beta) / \beta$, then for all 
$n_1, n_2$ satisfying $n_1>N^{\alpha_1}, $
$n_2 - n_{1} \geq N^{\alpha_1}$  we have

\begin{equation}
\label{KSK3d2}
 \sup_{x_0 \in \integers^d :|x_0|\leq N^{1/2+\delta} } |E_{n_1, n_2}(x_0)| \lesssim
 \sum_{j = 0}^d n_1^{\frac{d\beta -j}{2}}
(n_2 - n_{1})^{\frac{d\beta -2d + j}{2}} .
\end{equation}
Since the proof of \eqref{KSK3d2} 
is similar to that of \eqref{KSK3}, we only mention the main
difference. That is, now $j$ can take values $0,1,...,d$ and in dimension $d=1$ 
it could only take values $0,1$. This follows
from the fact that when applying $d$ summtions by parts to the function 
$H(n_1-n_{0}, x_1 - x_{0}) g_1(x_1)$, we obtain
$$
\nabla_1 ... \nabla_d (Hg) = 
\sum_{\{i_1,...,i_I\} \subset \{ 1,...,d\}} 
(\nabla_{i_1} ... \nabla_{i_I} H) (\nabla_{j_1}...\nabla_{j_J} g )
$$
where $\{j_1,...,j_J\} =  \{ 1,...,d\} \setminus \{i_1,...,i_I\}$.
Later we will need the following slight extension of \eqref{KSK3d2}.}

\begin{lemma}
\label{lemma:biggamma}
For any $a \in (0,1]$, if
$F \in \bG_{1/a}^\beta$, 
then for all 
$n_1, n_2$ satisfying $n_1 \geq N^{a}, $ 
$n_2 - n_{1} \geq N^{a}$  we have
\begin{equation}
\label{ManyDer}
 \sup_{x_0 \in \integers^d :|x_0|\leq N^{1/2+\delta} } |E_{n_1, n_2}(x_0)| \lesssim
 \sum_{j = 0}^d n_1^{\frac{d\beta -j}{2}}
(n_2 - n_{1})^{\frac{d\beta -2d + j}{2}} .
\end{equation}
\end{lemma}

\begin{proof}
Since the proof of the lemma
is similar to that of \eqref{KSK3}, we only mention the main
difference. That is, now $j$ can take values $0,1,...,d$ and in dimension $d=1$ 
it could only take values $0,1$. This follows
from the fact that when applying $d$ summations by parts to the function 
$H(n_1-n_{0}, x_1 - x_{0}) g_1(x_1)$, we obtain
$$
\nabla_1 ... \nabla_d (Hg) = 
\sum_{\{i_1,...,i_I\} \subset \{ 1,...,d\}} 
(\nabla_{i_1} ... \nabla_{i_I} H) (\nabla_{j_1}...\nabla_{j_J} g )
$$
where $\{j_1,...,j_J\} =  \{ 1,...,d\} \setminus \{i_1,...,i_I\}$.

In the proof of \eqref{KSK3}, we only 
used the definition of $\bG_\gamma^\beta$ for boxes with side length 
$L \geq N^{1/2 + \delta}$ 
(specifically in deriving \eqref{gammaapp}). 
In order to extend that proof to the present setting, we only need to replace 
$F(.)$ by $F(. - x_0)$. 
Thus assuming $\gamma > 1/a$ and using
the definition of $\bG_\gamma^\beta$, we can 
repeat the previous proof.
\end{proof}

Next we show that the extreme terms in the right hand side of \eqref{ManyDer} provide 
the main contribution.

Set $m_1=n_1,$ $m_2=n_2-n_1.$ 
 By Lemma \ref{lemma:biggamma}, we have for 
$m_1\geq N^a,$ $m_2 \ge N^a$, $|x_0|\leq N^{1/2+\delta}$ that
\begin{equation}
\label{ExpCases}
 |E_{n_1, n_2}(x_0)| \lesssim
 \begin{cases} m_1^{\frac{(\beta-1)d}{2}} m_2^{\frac{(\beta-1)d}{2}}
  & \text{if }m_2\geq m_1 \\[4pt]
  m_1^{\frac{\beta d}{2}} m_2^{\frac{d\beta-2d}{2}} & \text{if } m_2<m_1. 
 \end{cases}
\end{equation}
Note that the second bound is quite bad if $m_2\ll m_1.$ However we can improve it
by bootstrap. Namely we have

\begin{lemma}
If $N^a<m_2<m_1$ 
then 
$$ |E_{n_1, n_2}(x_0)| \lesssim m_2^{(\beta-1) d}$$
\end{lemma}
\begin{proof}
If $m_1\leq 2 m_2$ then the result follows from \eqref{ExpCases}. If $m_1>2 m_2$, 
let $k=m_1-m_2$
and note that $k > m_2$ and
$$ E_{n_1, n_2}(x_0)=\sum_{y\in \Z^d} H(k, y-x_0) E_{m_2, 2 m_2}(y). $$
The sum of the terms where $|y|>2 N^{1/2+\delta}$ decays faster than $N^{-r}$ 
for any $r$. The terms where $|y|\leq 2 N^{1/2+\delta}$ can be estimated 
by \eqref{ExpCases} with $m_1=m_2$ giving the result.
\end{proof}
We now combine the foregoing results in different regimes in case where $a=\eps.$ Then if 
$\gamma = 1/\eps$,
$F\in \bG_\gamma^\beta$ we gather that
\begin{equation}
\label{ExpCases2}
 |E_{n_1, n_2}(x_0)| \lesssim
 \begin{cases} m_1^{\frac{(\beta-1)d}{2}} m_2^{\frac{(\beta-1)d}{2}}
  & \text{if }m_2\geq m_1\geq N^\eps,\\[4pt]
m_2^{(\beta-1)d} & \text{if } m_1>m_2\geq N^\eps, \\[4pt]
1 & \text{if } \min(m_1, m_2)<N^\eps. 
 \end{cases}
\end{equation}
Summing the bounds of \eqref{ExpCases2} for $m_1, m_2\in \{1\dots N\}$ we obtain 
Proposition \ref{lemma:weak}.

\section{SLLN in dimension 1.}
\label{ScD=1}
\subsection{Reduction to occupation times sum.}
Here we prove Theorem \ref{ThD=1}.

Let $\ell_n(x)$ be time spent by the walker at site $x$ before time $n$. Set
$\DS \ell_\infty(x):= \lim_{n\to \infty} \ell_n(x).$
Thus $\ell_\infty(x)$ is the total time spent by the walker at site $x.$

\begin{lemma} \label{lem1}
  There exist $C,c>0,$ $\fp\in (0,1),$ $\eps_1$ such that for all $x \in \naturals$ and $m \in \naturals$
\begin{equation} \label{LTGeom}  
  \Prob( \ell_\infty(x) > m) < Ce^{-cm}. \end{equation}
  Furthermore 
  \begin{equation} \label{InEqRen}
  \left|\EXP (\ell_\infty(x))-\frac{1}{v}\right|\leq  \frac{C}{x^{\eps_1}},
  \quad \left|\Prob(\ell_\infty(x)=0)-\fp\right| \leq  \frac{C}{x^{\eps_1}}.
  \end{equation}
\end{lemma}

\begin{proof}
\eqref{InEqRen} follows from quantitative renewal theorem \cite{Rog73}. 
\eqref{LTGeom} holds since for $k\geq 1,$ 
$\DS \Prob(\ell_\infty(x)=k)=\Prob(\ell_\infty(x)\neq 0)\;\fp_0^{k-1} (1-\fp_0)$
where $\fp_0$ is the probability that $S_n$ returns to the origin at some positive moment of time. 
\end{proof}

Let $\DS \tT_N=\sum_{x=1}^N \ell_\infty(x) F(x).$ We will show that
with probability 1
\begin{equation}
\label{LLNOcc}
\frac{\tT_N}{N}\to \frac{\brF_+}{v}. 
\end{equation}

We first deduce Theorem \ref{ThD=1} from \eqref{LLNOcc} and then prove \eqref{LLNOcc}.
Denote $\DS \bbL_N=\sum_{x=1}^N \ell_\infty(x).$ By the strong law of large numbers
for~$S_N$

\begin{equation}
\label{LLNLocTime}
\frac{\bbL_N}{N}\to \frac{1}{v}.
\end{equation}
On the other hand for each $\eps$  and for almost every $\omega$,
there is some $N_0 = N_0(\eps, \omega)$ so that for all $N>N_0$,
$$ \left|T_N-\tT_{Nv(1-\eps)}\right|\leq ||F||_\infty\left(\bbL^-+[\bbL_{Nv(1+\eps)}-\bbL_{Nv(1-\eps)}]\right),$$
where $\DS \bbL^-=\sum_{x=-\infty}^0 \ell_\infty(x)$ is the total time spent on the negative halfline.
In view of \eqref{LLNLocTime},
$ \frac{T_N-\tT_{Nv(1-\eps)}}{N} $ can be made as small as we wish by taking $\eps$ small. Hence
Theorem \ref{ThD=1} follows from \eqref{LLNOcc}.

In order to prove \eqref{LLNOcc} we observe that by Lemma \ref{lem1}
$$ \EXP\left(\frac{\tT_N}{N}\right)=\frac{1}{N}\sum_{x=1}^N F(x) \EXP(\ell_\infty(x))= 
\frac{1}{v} \left[\frac{1}{N} \sum_{x=1}^{N}  F(x)\right]+O\left(N^{-\eps_1}\right)=
\frac{\brF_+}{v} + o(1), $$
as $N\to \infty$.

We need the following bound, which will be proved in \S \ref{SSCovOcc}.

\begin{lemma}
\label{LmLTCov}
There are constants $C$ and $\eps_2$ such that
for each $n_1<n_2$
$$\left|\Cov(\ell_\infty(n_1), \ell_\infty(n_2))\right|\leq 
C\left(\frac{1}{n_1^{\eps_2}}+\frac{1}{(n_2-n_1)^{\eps_2}}\right). $$
\end{lemma}

Lemma \ref{LmLTCov} implies that 
$$\Var\left(\frac{\tT_N}{N}\right)\leq \frac{C}{N^{\eps_2}}$$
and so
$$ \Prob\left(\frac{|\tT_N-\EXP(\tT_N)|}{N}\geq \delta\right)\leq \frac{C}{\delta^2 N^{\eps_2}}. $$
Set $r=2/\eps_2.$ By Borel-Cantelli Lemma 
$$\frac{\tT_{n^r}}{n^r}\to \frac{\brF}{v} \text{ as } n\to \infty $$
almost surely. On the other hand,
\eqref{LTGeom} and the Borel-Cantelli Lemma imply that, 
with probability 1, for all sufficiently large $x,$ $\ell_\infty(x)\leq \ln^2 x.$ 
Given $N,$ take $n$ such that
$n^r \leq N<(n+1)^r.$ Then
$$ \left| \tT_N-\tT_{n^r} \right| \leq ||F||_\infty \left(\bbL_N-\bbL_{n^r}\right)
\leq C N^{(r-1)/r} \ln^2 N .$$
It follows that 
$\DS \frac{\tT_N}{N}=\frac{\tT_{n^r}}{n^r}+
o(1),
$
for $N\to\infty$, 
proving \eqref{LLNOcc}.

\subsection{Covariance of occupation times.}
\label{SSCovOcc}
The proof of Lemma \ref{LmLTCov} relies on the following estimates.

\begin{lemma} 
\label{LmLDMin}
There are constants
$C$ and $ \eps_3$ such that
for all $m \ge 1$,
$$\Prob(\min_n(S_n)\leq -m)\leq \frac{C}{m^{\eps_3}}. $$
\end{lemma}

Lemma \ref{LmLDMin} (with $\eps_3 = \beta -1$) follows from Theorem 2(B) of \cite{V77}.

\begin{lemma}
\label{LmLTMC}
For each $\delta>0$ there is a constant $C(\delta)$ such that the following holds.
Consider a Markov chain with states $\{1, 2, 3\}$ and transition matrix
$$ \left(\begin{array}{ccc} p_1 & q_1 & \eta_1\\
q_2 & p_2 & \eta_2\\
0 & 0 & 1\end{array}\right)$$
and initial distribution $(\pi_1, \pi_2, \pi_3).$ Assume that
\begin{equation}
\label{Ell}
 q_1>\delta, \quad  \text{ and } \quad \eta_2 > \delta.
\end{equation}

Let $\fl_1$ and $\fl_2$ denote the occupation times of sites 1 and 2. Then
$$ \left|\Cov(\fl_1, \fl_2) \right|\leq C(\delta)\left( \frac{q_1}{q_1+\eta_1}(1-\pi_1)-\pi_2+q_2\right).$$
\end{lemma}

In the special case where $\eta_1=0$, $\pi_1=1,$
Lemma \ref{LmLTMC} follows from
\cite[Lemma 3.9(a)]{DG12}.
In this case the statement simplifies significantly since the first term in the RHS vanishes.

The proof of Lemma \ref{LmLTMC} will be given
in \S \ref{SS3State}.

We apply Lemma \ref{LmLTMC} to the states $(n_1, n_2, \infty)$
with $n_1 < n_2$.
This means that we define a 3-state Markov chain as a function of the
random walk $( S_k )$, such that the chain starts in the state 1, 
if the random walk visits $n_1$ for the first time before it visits $n_2$; or in the state
2, if the random walk visits $n_2$ for the first time before it visits $n_1$; or in 
in the state 3 if the random walk never visits $n_1$ or $n_2$. After
that, the chain transitions to state 1, 2 or 3, respectively, if the
next return of the random walk to the set $\{ n_1, n_2 \}$ occurs at $n_1$,
$n_2$, or never does. Clearly 3 is an absorbing state for this chain. So
$$\pi_1=\Prob(n_1\text{ is visited before } n_2), \quad
\pi_2=\Prob(n_2\text{ is visited before } n_1), $$
$$\pi_3=\Prob(n_1\text{ and } n_2\text{ are not visited}).$$
Let $V_n$ be the event that $n$ is visited by our random walk. 
Note that Lemma \ref{LmLDMin} implies $q_2 =O\left((n_2-n_1)^{-\eps_3}\right).$ Hence
the probability that both $n_1$ and $n_2$ are visited with $n_2$ being 
the first
is also $O\left((n_2-n_1)^{-\eps_3}\right).$ 
Therefore
$$
\pi_2 = \Prob (V_{n_2}) - \Prob(n_1 \text{ visited first, then } n_2 \text{ is visited})
\asymp \Prob (V_{n_2})-
\Prob (V_{n_1} \cap V_{n_2}),
$$
where $\asymp$ means the difference between the LHS and the RHS is 
$$ O\left(n_1^{-\eps_2}\right)+O\left((n_2-n_1)^{-\eps_2}\right)\quad\text{where}\quad
\eps_2=\min(\eps_1, \eps_3). $$
Likewise,
$$ \pi_1\asymp \fq, \quad \frac{q_1}{q_1+\eta_1}=\Prob_{n_1}(V_{n_2})\asymp \fq, 
$$
where 
$\fq = 1- \fp$ (see \eqref{InEqRen}). Combining the last two displays, we obtain
$$
\quad \pi_2\asymp \Prob(V_{n_2})-\Prob(V_{n_1} \cap V_{n_2})\asymp \fq-\fq^2, 
$$
These estimates, combined with
Lemma \ref{LmLTMC} imply Lemma \ref{LmLTCov}.

\subsection{Analysis of three state chains.}
\label{SS3State}
\begin{proof}[Proof of Lemma \ref{LmLTMC}]
  Under the assumptions of the lemma,
  $\fl_1, \fl_2$ and $\fl_1\fl_2$ are uniformly integrable
(the uniformity is over all chains satisfying \eqref{Ell}). Let $\bar\fl_1$
be the time spent at 1 before the first visit to another state.
Then,
by the uniform integrability,
$$ \EXP(\fl_1-\bar\fl_1)=O(q_2), \quad \EXP((\fl_1-\bar\fl_1)\fl_2)=O(q_2), \quad
\EXP_2(\fl_1)=O(q_2), \quad \EXP_3(\fl_j)=0 .$$
Hence
$$ \EXP(\fl_1 \fl_2)=\pi_1 \EXP_1 (\bar\fl_1) \frac{q_1}{q_1+\eta_1} \EXP_2(\fl_2)+O(q_2); $$
$$ \EXP(\fl_1)=\pi_1 \EXP_1 (\bar\fl_1)+O(q_2) \quad\text{and}\quad
\EXP(\fl_2)=(\pi_1\frac{q_1}{q_1+\eta_1}+\pi_2) \EXP_2(\fl_2)+O(q_2).
$$
Therefore
\begin{align*} \Cov(\fl_1, \fl_2)&=&\pi_1 \EXP_1 (\bar\fl_1) \frac{q_1}{q_1+\eta_1} \EXP_2(\fl_2)
-\pi_1 \EXP_1 (\bar\fl_1) (\pi_1\frac{q_1}{q_1+\eta_1}+\pi_2) \EXP_2(\fl_2)+O(q_2) \\
&=&\pi_1 \EXP_1 (\bar\fl_1) \EXP_2(\fl_2) \left[\frac{q_1}{q_1+\eta_1} (1-\pi_1)-\pi_2\right]
+O(q_2) \hskip3cm
\end{align*}
as claimed.
\end{proof}

\section{Counterexamples to the strong law.}
\label{sec:oceans}
Here we prove Theorem \ref{ThOcean}. 

Consider first the case $d=1$. 
Assume that we are given a strictly increasing sequence $d_k \in  
\naturals$, $d_1 = 1$, $d_k \to \infty$ sufficiently quickly. 
We will impose finitely
many lower bounds on the growth of $d_k$ and so one can take the biggest lower 
bound.
Now, for $l = d_k, d_k+1,..., d_{k+1} -1$, define
$a_l= k$.
Notice that there is a 
bijective correspondence between strictly increasing sequences $(d_n)$ and 
diverging sequences $(a_n)$ such that $a_n-a_{n-1} \in \{0,1\}$.

Next,  define
$b_1 \gg 1$, $b_{n+1} = b_n + \lfloor b_n / a_n \rfloor$. By induction, we see that 
\begin{equation}
\label{eq:b_nbounds}
a_n\leq n < b_n < b_{n+1} \quad \text{and, for large }n,\quad b_n< 2^{n}.
\end{equation}
Consider the function $F$ defined by $F(0) = 0$,
\begin{equation}
\label{defcounterex}
F(x) = \begin{cases}
1 & \text{ if } b_{2k} \leq x < b_{2k+1} \text{ for some } k\\
0 & \text{ if } b_{2k +1} \leq x < b_{2k+2} \text{ for some } k\\
\end{cases}
\end{equation}
for $x>0$ and $F(x) = F(-x)$ for $x <0$. We start by
verifying that $F \in \bG_0$.

\begin{lemma}
If 
\begin{equation}
\label{eq:dk}
\lim_{k \to \infty} \sqrt{d_k}/d_{k-1} = \infty,
\end{equation}
then $F \in \bG_0$ with $\bar F = 1/2$.
\end{lemma}

Note that \eqref{eq:dk} is satisfied if e.g. $d_1=1$ and 
$d_k=2^{2^{2^k}}$ for $k \geq 2$.

\begin{proof}
Recall that 
$b_n < 2^n$, for large $n$ (say for $n\geq n_0$).
Next, we establish the a priori lower bound
\begin{equation}
\label{eq:bnlowerbd}
2^{c \sqrt n}\leq b_n 
\end{equation}
for some $c$. To prove \eqref{eq:bnlowerbd}, first observe that, 
by \eqref{eq:dk}, there is some $C$
so that for all $n \geq n_0$,
$a_n < C \sqrt n$.
Thus
$$
b_{n+1} \geq b_n + \lfloor b_n / C \sqrt n \rfloor \geq b_n(1 + 1/(C\sqrt n)) - 1 
\geq b_n(1 + 1/(2C \sqrt n))
$$
where the last inequality follows from $n < b_n$. Thus 
$\DS b_n \geq b_{n_0} 
\prod_{k={n_0}}^{n-1} (1 + 1/(2C \sqrt k))$. Hence
$$b_n \geq \exp\left[ \ln (b_{{n_0}}) + \sum_{k=1}^n \ln(1 + 1/(2C \sqrt k)) 
\right]
\geq 2^{c \sqrt n}$$
for $n \geq n_0$. 
Clearly, we can assume that \eqref{eq:bnlowerbd} holds for all $n < n_0$
by further decreasing
$c$ if necessary.

We will prove that $\DS \lim_{m\to \infty} \cS(m)/m = 0$, where 
$\cS(m) = \sum_{x=1}^{m}(F(x) - 1/2)$. This will imply the lemma. Now given 
$m$, let $n'$ be so that $b_{n'-1} \leq m < b_{n'}$. 
Let us denote 
\begin{align*}
n'' = \max \{ n \leq n': a_n = a_{n-1} + 1\}\\
n''' =\max \{ n < n'': a_n = a_{n-1} + 1\} 
\end{align*}
which are well defined for $m$ large enough.
Assume without loss of generality that $n'$
is even (the case of odd $n'$ is similar). Then
$$
\cS(m) \leq \frac12 b_{n'''} + 
\sum_{x = b_{n'''}+1}^{b_{n''}} \left( F(x) - \frac12 \right) 
+ \sum_{x = b_{n''}+1}^{b_{n'}} \left( F(x) - \frac12 \right) 
=:  \frac12 b_{n'''} +  \cS_1 + \cS_2
$$
Note that we have
$$
\cS_1 = \frac12 \sum_{n=n'''}^{n''-1} (-1)^n (b_{n+1} - b_n).
$$
We claim that
\begin{equation}
\label{eq:csbound}
|\cS_1| \leq \frac12 (b_{n''} - b_{n'' -1}).
\end{equation}
To prove \eqref{eq:csbound}, observe that 
for $n=n''', n'''+1,...,n''-1$, the sequence $a_n$ is constant and hence
the sequence 
$n \mapsto (b_{n+1} - b_n)$ is monotone increasing for this range of $n$'s. 
So in case $n''-1$ is even, we have
$$
\frac12 \sum_{n=n'''}^{n''-2} (-1)^n (b_{n+1} - b_n) \leq 0 \leq \cS_1.
$$
Hence
$$0 \leq \cS_1 \leq \cS_1 - \frac12 \sum_{n=n'''}^{n''-2} (-1)^n (b_{n+1} - b_n)
= \frac12( b_{n''} - b_{n'' -1})$$
(the case when $n''-1$ is odd is similar). We have verified \eqref{eq:csbound}.
Likewise, we have
$$
|\cS_2| \leq \frac12 (b_{n'} - b_{n' -1}).
$$
Consequently, 
$$
\frac{\cS(m)}{m} \leq \
\frac{b_{n'''}}{2b_{n'}} + \frac{b_{n''} - b_{n''-1}}{2b_{n'}} 
+ \frac{b_{n'} - b_{n'-1}}{2b_{n'}}.
$$
Since $b_{n'-1} \sim m \sim b_{n'}$, the last two terms on the RHS converge to zero. 
To estimate the first term, we first use \eqref{eq:bnlowerbd} to derive
$$
\frac{b_{n'''}}{b_{n'}} \leq \frac{b_{n'''}}{b_{n''}} 
 <  2^{n''' - c\sqrt{n''}}.
$$
Now by the definition of $n''$ and $n'''$, there is some $k_0$ so that $n'' = d_{k_0}$
and $n''' = d_{k_0 -1}$ and so the RHS of the last displayed inequality
also converges to zero as $m \to \infty$ (and consequently $k_0 \to \infty$)
by \eqref{eq:dk}. The lemma follows.
\end{proof}

Let us denote $c_n = (b_n + b_{n+1})/2$, $t_n = \lfloor c_n^{\alpha} \rfloor$ and
$$I_n = [c_n - \lfloor b_n / 4 a_n \rfloor, c_n + \lfloor b_n / 4 a_n \rfloor].$$
We will show that almost surely, infinitely many of the events
$$
A_{2n} = \{ \forall k \in[ t_{2n}, 3 t_{2n}]: S_k \in I_{2n}\}
$$
occur and, likewise, infinitely many of the events 
$$
A_{2n+1} = \{ \forall k \in [t_{2n+1}, 3 t_{2n+1}]: S_k \in I_{2n+1}\}
$$
occur. This 
proves the theorem as $A_{2n}$ implies $T_{3t_{2n}} \geq 2 t_{2n}$ and 
$A_{2n+1}$ implies $T_{3t_{2n+1}} \leq t_{2n+1}$.

To complete the proof, let us fix a sequence $D_n \nearrow \infty$ such that
$$ \sum_n \Prob(|S_n|\geq D_n)<\infty. $$
Then by the Borel-Cantelli Lemma,
$$
\mathbb P (\exists N: \forall n > N: |S_n| < D_n) =1.
$$
Now we choose a subsequence $n_k \in \integers$ inductively so that 
\begin{equation}
\label{n_kparity}
n_{k+1} \equiv n_k  \pmod{2}
\end{equation}
and
\begin{equation}
\label{eq:nkbound}
n_{k+1} > \max \left\{\exp \left(D_{  \lceil 2^{\alpha (n_k +1)} \rceil }\right), 
\exp \left( \frac{1}{\alpha} 2^{(n_k +1)\alpha}\right) \right\}. 
\end{equation}
These bounds, combined with \eqref{eq:b_nbounds}, give
\begin{equation}
\label{eq:b_nbd2}
b_{n_{k+1}} > \exp (D_{t_{n_k}}) \text{ and } t_{n_{k+1}} > \exp (t_{n_k}).
\end{equation}

We want
to show that for every $\eps >0$ and every $K$,
\begin{equation}
\label{eq:BC}
\mathbb \Prob \left(\bigcap_{k=K}^{\infty} A_{n_k}^c\right) < \eps.
\end{equation}

Since $\eps >0$ is arbitrary, it follows that infinitely many of the events $A_{n_k}$ happen.

Choosing $n_k$ to be even for all $k$ we see
that almost surely infinitely many of the events $A_{2n}$
and happen. Likewise, choosing $n_k$ to be odd for all $k$ we see that
infinitely many of the events $A_{2n+1}$ happen.
Thus it remains to verify \eqref{eq:BC}.

Given $\eps$ and $K$, choose $K'>K$ so that 
$\mathbb P (\mathcal B) < \eps $, where 
$$\mathcal B = \{ \exists n \ge n^{\alpha}_{K'}: |S_n| \ge D_n\}.$$
Then we write
\begin{equation}
\label{eq:BCbound}
\P \left(\bigcap_{k=K}^{\infty} A_{n_k}^c\right) \leq 
\P \left(\bigcap_{k=K'}^{\infty} A_{n_k}^c\right)
\leq \eps + \P \left(\bigcap_{k=K'}^{\infty} A_{n_k}^c \cap \mathcal B^c\right).
\end{equation}
By construction, we have
\begin{gather}
\P \left( A_{n_{k+1}}^c \cap \mathcal B^c \Big| \bigcap_{j=K'}^{k} A_{n_j}^c \cap \mathcal B^c\right)
\leq 1 - \P \left( A_{n_{k+1}} \Big| \bigcap_{j=K'}^{k} A_{n_j}^c \cap \mathcal B^c\right) \nonumber\\
\leq 1- \min_{x: |x| < D_{ 3t_{n_k}}} \P ( A_{n_{k+1}} | S_{ 3t_{n_k}} = x) \label{Anest}
\end{gather}
for $k > K'$.
\begin{lemma}
\label{lem:stableprocess}
There is a constant $K_0$ and a sequence $a_n \nearrow \infty$ with $a_n 
- a_{n-1} \in \{ 0,1\}$
such that for any $k \geq K_0$,
$$
\min_{x: |x| < D_{ 3t_{n_k}}} \P ( A_{n_{k+1}} | S_{3t_{n_k}} = x) 
\geq \frac{1}{k} .
$$
Also, $(a_n)$ is such that the corresponding $(d_n)$ satisfies \eqref{eq:dk}.
\end{lemma}
Clearly, Lemma \ref{lem:stableprocess} 
combined with \eqref{Anest} and \eqref{eq:BCbound} implies \eqref{eq:BC}.
Thus the proof of Theorem \ref{ThOcean}
for the case $d=1$ will be completed once we prove Lemma \ref{lem:stableprocess}.

\begin{proof}[Proof of Lemma \ref{lem:stableprocess}]
Recall that the invariance principle gives
\begin{equation}
\label{invprinciple}
\frac{S_{\lfloor Nt \rfloor }}{N^{1/\alpha}} \Rightarrow Y_t,
\end{equation}
where $Y_t$ is a stable L\'evy process. In particular, 
$Y_1$ is a stable random variable with parameter $\alpha$ and "skewness" $\beta \in [-1,1]$
(see e.g. \cite{B96}, Chapter VIII).
Now we distinguish two cases.

{\bf Case 1} $\alpha >1$ or $|\beta|\neq 1.$
The proof in case 1 consists of 3 steps.

{\it Step 1}: We prove that $q >0$ where
\begin{equation}
\label{DefQ}
q = \inf_{y \in [-1/16, 1/16]} 
\mathbb P \left(\sup_{t \leq 1} |Y_t| < \frac14, |Y_1| < \frac{1}{16} \Big| Y_0 = y \right).
\end{equation}
Indeed,
as $\alpha >1$ or $|\beta|\neq 1$, the stable process $Y_t$ cannot be a subordinator. In particular, the 
density of $Y_t$ is positive everywhere for every $t>0$ and $Y_t$ has the scaling property (see page 216
in \cite{B96}). Thus we have
$$
\liminf_{\eps \searrow 0} \inf_{y \in [-1/16, 1/16]} 
\mathbb P \left(|Y_\eps| < \frac{1}{16} \Big| Y_0 = y \right) = p >0.
$$
By Exercise 2 of Chapter VIII in \cite{B96}, there is some $\eps >0$ such that 
$$
\mathbb P \left(\sup_{t \leq \eps} |Y_t| < \frac{3}{16} \Big| Y_0 = 0 \right) >1-p/2.
$$
Combining the last two displayed equations, we derive
$$
\inf_{y \in [-1/16, 1/16]} 
\mathbb P \left(\sup_{t \leq \eps} |Y_t| < \frac14, |Y_\eps| < \frac{1}{16} \Big| Y_0 = y \right) \geq p/2.
$$
Applying this inequality inductively, we obtain that
$q \geq (p/2)^{\lceil 1/\eps \rceil}$, which completes Step 1.

{\it Step 2}: We prove that there is some $\brc$ so that for every $k$
\begin{equation}
\label{lltarrival}
\min_{x: |x| < D_{ 3t_{n_k}}} \Prob \left( \left| S_{t_{n_{k+1}}} - c_{n_{k+1}} \right| < \frac{b_{n_{k+1}}}{16 a_{n_{k+1}}}
\Big| S_{ 3t_{n_k}} = x \right) \geq \frac{\brc}{a_{n_{k+1}}}.
\end{equation}
To simplify the formulas, let us write
$$
\ta = a_{n_{k+1}},\; \; \tb = b_{n_{k+1}}, \;\; \tc = c_{n_{k+1}},\; \;  \tilde t = t_{n_{k+1}}.
$$
Recall that the density of $Y_1$ is strictly positive. The LLT now implies that for every $K >0$ there is some $c$ so that for all $z$ with $|z| \leq K$,
$$
\Prob(S_N = \lfloor z N^{1/\alpha} \rfloor) \geq \frac{c}{N^{1/\alpha}}.
$$
Next observe that by the definition of $\ta, \tb, \tc$, $\tilde t$ and by \eqref{eq:b_nbd2}
$$
t_{n_{k+1}} - 3 t_{n_{k}} \sim \tilde t \sim \tb^{\alpha}, 
\quad \tilde c \pm D_{3 \tilde t} \sim \tilde t^{1/\alpha}
\text{ as } k \to \infty.
$$
Thus for every $x \in \integers$ with $|x| <D_{3 \tilde t}$,
$$
 \Prob \left( \left| S_{\tilde t} - \tc \right| < \frac{\tb}{4 \ta}
\Big| S_{ 3t_{n_k}} = x \right) 
= \Prob \left( \left| S_{\tilde t - 3t_{n_k}} - (\tc - x) \right| < \frac{\tb}{4 \ta} \right) 
\geq \frac{c}{\tb} \frac{\tb}{4 \ta} = \frac{\brc}{\ta}.
$$

{\it Step 3}: We prove that there is some 
${ \brp}\in (0,1)$ such that
\begin{equation}
\label{eq:case1exp}
\min_{x: |x| < D_{3t_{n_k}}} \Prob ( A_{n_{k+1}} | S_{3t_{n_k}} = x) 
\geq  \frac{\brc}{a_{n_{k+1}}} \;{ \brp}^{a_{n_{k+1}}^\alpha}.
\end{equation}
The lemma will follow from \eqref{eq:case1exp} as we can assume that 
our sequence $a_n \nearrow \infty$ with $a_n - a_{n-1} \in \{ 0,1\}$
satisfies
$a_{n_k} \leq (- \log k / \log \tilde p )^{1/\alpha}$
for a fixed $\tilde p \in (0,{ \brp})$.

To prove \eqref{eq:case1exp}, we will combine \eqref{DefQ} with \eqref{lltarrival}.
Namely, define
$$
\tilde A =A_{n_{k+1}}, \quad \tilde N = \left\lceil \left( \frac{\tb}{\ta} \right)^{\alpha}
\right\rceil, \quad \tilde l = \lceil 2 \ta^{\alpha} \rceil + 1. 
$$
Let us define the event $\tilde B_0 = \{  |S_{\tilde t } - \tc |< \frac{\tilde N^{1/\alpha}}{16}\}$ and for $\brl = 1,...,\tl $:
\begin{align*}
\tilde B_{\brl} 
= 
\tilde B_0 \cap
 \Bigg\{ & \forall  l = 0,...,\brl -1, \forall m = 1,..., \tilde N -1: \\
&
|S_{\tilde t + l \tilde N + m}  - \tc |< \frac{\tilde N^{1/\alpha}}{4},\;
|S_{\tilde t + (l+1) \tilde N} - \tc |< \frac{\tilde N^{1/\alpha}}{16} \Bigg\}.
\end{align*}
Note that $\tilde B_{\tl}$ implies $\tilde A$.
Thus it is enough to prove \eqref{eq:case1exp}
with $\tilde A$ replaced by $\tilde B_{\tl}$. 
 First, \eqref{lltarrival} implies
$$
\min_{x: |x| < D_{3t_{n_k}}} \Prob ( \tilde B_{0} | S_{3t_{n_k}} = x) \geq
\frac{\brc}{\ta}.
$$
Next, we derive
$$
\min_{x: |x| < D_{3t_{n_k}}} \Prob ( \tilde B_{1} | S_{3t_{n_k}} = x) \geq $$
$$
\min_{x: |x| < D_{3t_{n_k}}} 
\Prob ( \tilde B_{0} | S_{3t_{n_k}} = x)
\min_{y: |y | < \frac{\tilde N^{1/\alpha}}{16}}
\Prob \left( \max_{m=1,..., \tilde N -1} |S_{m}|< \frac{\tilde N^{1/\alpha}}{4}, 
|S_{\tilde N}| <  \frac{\tilde N^{1/\alpha}}{16} \big| S_0 = y\right) 
$$
$$ \geq \frac{\brc}{\ta} \frac{q}{2},$$
where the first inequality follows from the Markov property and the second one 
follows from \eqref{invprinciple} and \eqref{DefQ}.
Now we can apply an induction on $\brl$, to obtain
$$
\min_{x: |x| < D_{3t_{n_k}}} \Prob ( \tilde B_{\brl} | S_{3t_{n_k}} = x) \geq
\frac{\brc}{\ta} \left( \frac{q}{2} \right)^{\brl}.
$$
In particular,
$$
\min_{x: |x| < D_{3t_{n_k}}} \Prob ( \tilde B_{\tl} | S_{3t_{n_k}} = x) \geq
\frac{\brc}{\ta} \left( \frac{q}{2} \right)^{\tilde l}
$$
which completes the proof of Step 3.

{\bf Case 2} $\alpha <1$ and $|\beta|=1$.
Let us assume  $\beta = 1$ (otherwise apply the forthcoming argument to $-X_i$). 

\eqref{lltarrival} still holds in case 2, however a new approach is required to estimate
$\Prob ( A_{n_{k+1}} | S_{t_{n_k}} = x) $
since now the process $Y_t$ 
(a.k.a. stable subordinator)
is non-decreasing and thus $q=0$.
Note however that in case 2, $\sup_{t \leq 1} |Y_t| = Y_1$ and thus it suffices to estimate one 
random variable instead of a stochastic process.
Recall that the density of $Y_1$ is strictly positive on $\mathbb{R}^+$. 
Thus applying \eqref{invprinciple}
to $|X_i|$ (which is also in the standard domain of attraction of the totally 
skewed $\alpha$-stable distribution)
we obtain the following: for any $\eps >0$ there exist $N_0(\eps)$ and $\delta(\eps)>0$ 
such that for any $N \geq N_0(\eps)$,
\begin{equation}
\label{case2pos}
\P\left( \sum_{n = 1}^{3 N} |X_n| \leq \frac{\eps}{8} N^{1/\alpha} \right) > \delta(\eps).
\end{equation}
Without loss of generality, we assume that $N_0$ and $\delta$ are, respectively,
non-decreasing and non-increasing functions of $\eps$.
Now we define the sequence $a_n$ inductively. First, let $a_1 = 1$. Now assume that $a_{n_{k}}$ is defined.
Let $a_m = a_{n_{k}}$ for $m = n_{k} + 1,..., n_{k+1} -1$. Next, we define $a_{n_{k+1}} = a_{n_{k}}+1$
if both of the following conditions are satisfied:
\begin{enumerate}
\item[(A)] $ N_0\left(\frac{1}{a_{n_{k}} +1}\right)<n_{k+1}^{\alpha} $ and
\item[(B)] $\frac{\brc}{2a_{n_{k}} +1} \;\delta \!\left( \frac{1}{a_{n_{k}} +1}\right) > \frac{1}{k+1}$.
\end{enumerate}
Here, $\brc$ is the constant from \eqref{lltarrival}.
If either (A) or (B) fails, we put $a_{n_{k+1}} = a_{n_{k}}$. 
Note that by \eqref{eq:nkbound}, the resulting sequence $d_k$ satisfies
\eqref{eq:dk}.

Observe that by our construction, for all $k$, we have 
\begin{equation}\label{NZero}
 N_0\left(\frac{1}{a_{n_{k}}}\right)<n_{k}^{\alpha} ;
 \end{equation}
\begin{equation} \label{Delta}\frac{\brc}{2 a_{n_{k}}} \;\delta \!\left( \frac{1}{a_{n_{k}}} \right) > \frac{1}{k}.
\end{equation}
Indeed,
if $a_{n_{k+1}}=a_{n_k}+1$ then \eqref{NZero} and \eqref{Delta} follow from conditions (A) and (B) above.
If $a_{n_{k+1}}=a_{n_k}$ then \eqref{NZero} and \eqref{Delta} follow by induction since the LHSs
of both \eqref{NZero} and \eqref{Delta} do not change when we replace $k$ by $k+1,$ while
the RHS of \eqref{NZero} increases and the RHS of \eqref{Delta} decreases.

By construction, $a_n \nearrow \infty$. Let $K_0$ be the smallest integer $k$ so that $a_{n_k} = 2$.
We prove that the lemma holds with this choice of $K_0$ and $a_n$. 
Recall that 
 by \eqref{eq:b_nbounds},
$b_{n_k}>n_k$ and so by \eqref{NZero}, $N:= b_{n_k}^{\alpha} > N_0(\eps)$ with $\eps = 1/a_{n_k}$. Applying
\eqref{case2pos} with this $N$ and $\eps$ and using \eqref{Delta}, we obtain
\begin{equation*}
\P \bigg( \left| S_m - S_{t_{n_{k}}} \right| < 
\frac{b_{n_{k}}}{8 a_{n_{k}}} \;\; \forall m \leq 3 b^{\alpha}_{n_{k}} 
\bigg) 
> \frac{a_{n_k}}{\brc} \frac{2}{k}.
\end{equation*}
and since $t_{n_k} < b_{n_k}^{\alpha}$ and $k-1 > k/2$, we arrive at
\begin{equation}
\label{eq:stayintube}
\P \bigg( \left| S_m - S_{t_{n_{k}}} \right| < 
\frac{b_{n_{k}}}{8 a_{n_{k}}} \;\; \forall m \leq  3t_{n_{k}} 
\bigg) 
> \frac{a_{n_{k-1}}}{\brc} \frac{1}{k-1}.
\end{equation}
As discussed above, \eqref{lltarrival} holds. 
Since $a_{n_{k+1}} - a_{n_k} \in \{0,1\}$ and $a_{n_k} \nearrow \infty$,
we can also assume 
 changing $\brc$ if necessary that \eqref{lltarrival} holds with
$a_{n_{k+1}}$ replaced by $a_{n_k}$ on the RHS.
Now we use the Markov property, 
\eqref{eq:stayintube} with $k$ replaced by $k+1$
and \eqref{lltarrival} (with the updated RHS) to derive
\begin{align*}
\min_{x: |x| < D_{ 3t_{n_k}}} & \P ( A_{n_{k+1}} | S_{3t_{n_k}} = x)  \\
\geq &
\min_{x: |x| < D_{ 3t_{n_k}}} \Prob \left( \left| S_{t_{n_{k+1}}} - c_{n_{k+1}} \right| < \frac{b_{n_{k+1}}}{16 a_{n_{k+1}}}
\Big| S_{ 3t_{n_k}} = x \right) \times\\
&\times \P \bigg( \left| S_m - S_{t_{n_{k+1}}} \right| < 
\frac{b_{n_{k+1}}}{8 a_{n_{k+1}}} \;\; \forall m \leq  3t_{n_{k}} 
\bigg) \\
\geq &\frac{\brc}{a_{n_{k}}} \frac{a_{n_{k}}}{\brc} \frac{1}{k} = \frac{1}{k}.
\end{align*}
The lemma follows.
\end{proof}

The above proof, with a few minor adjustments, applies to arbitrary dimension $d$. 
Specifically, we need to consider the function 
$\mathcal F \in \bG_0$
defined by
$$\mathcal F(x_1,...,x_d) = 
\begin{cases}
F(x_1) & \text{ if } |x_i| \leq |x_1| \text{ for } i=2,...,d\\
\frac12 & \text{ otherwise,}
\end{cases}
$$
where $F$ is given by \eqref{defcounterex} and we need to replace $I_n$ by
$$ [c_n - \lfloor b_n / 4 a_n \rfloor, c_n + \lfloor b_n / 4 a_n \rfloor]
\times \left[- \lfloor b_n / 4 a_n \rfloor, \lfloor b_n / 4 a_n \rfloor\right]^{d-1}.
$$

\begin{remark}
It is easy to adjust the above proof to derive the following stronger version of Theorem \ref{ThOcean}:
There is a function $F \in \bG_0$ so that $F$ only takes values $\{0,1\}$, $\brF = 1/2$ and for almost every $\omega$ and for any
$a \in [0,1]$,
there is a subsequence $n_k = n_k(a, \omega)$ such that 
$T_{n_k} / n_k \to a$.
\end{remark}

\section{Conclusions.} The results proven in this paper show that
for random walks the weak law of large numbers holds in the largest possible space of global
observables, namely $\bG_0.$ On the other hand, the strong law of large numbers fails, 
for general $\bG_0$ observables,
except for the walks with drift in dimension 1. In that case the path of the walk is almost deterministic 
and so the ergodic theory for occupation times could be used. The good news
is that the weak law of large numbers seems to be a good setting for homogenization theory
(cf Theorem~\ref{ThArc}), so the space 
$\bG_0$ could be useful for that purpose. If we have some control on fluctuations over the mesoscopic 
scale as provided, for example, by the space $\bG_\gamma,$ then we can ensure the strong law. 
If we have polynomial control on the mesoscopic scale, as provided by the space $\bG_\gamma^\beta$
then we can estimate the rate of convergence. In particular, our results give optimal rate of convergence
for two important special cases: random walks in random scenery and quasi-periodic observables.

We note that the main ingredient in most proofs is local limit theorem and its extensions, such as the 
Edgeworth expansion used in Section \ref{ScSpeed}. This makes it plausible that similar results hold for
other systems where the local limit theorem holds, including the systems described in 
\cite{BCR, CD15, DG13, DG18, DN-LLT, DN-Inf, DN18}. Another natural research direction motivated by the
present work is limit theorems for global observables. It is likely that just assuming that $F$ belongs to an
appropriate $\bG_*$
will not be enough to derive limit theorems.
For example, the computation of variance for $T_N$ done in Sections \ref{ScWeak}
and \ref{ScSpeed} involve the expression of the form $\tF_z(x):=F(x) F(x+z)$ for a fixed $z\in \Z^d.$
Therefore additional restrictions seem to be required to obtain limit theorems.

Extending our results to more general systems as well as limit theorems for global observables will
be the subject of future work.

\section*{Acknowledgements.} The research of DD was partially sponsored by NSF DMS 1665046. The research of ML was partially supported by PRIN 2017S35EHN (MUR, Italy). The research of PN was partially sponsored by NSF DMS 1800811 and NSF DMS 1952876.

\end{document}